\documentclass[11pt,a4paper]{amsart}
\usepackage{amssymb}
\usepackage[color,all]{xy}
\usepackage{wrapfig}
\usepackage{secdot}
\usepackage[dvipdfmx]{hyperref,graphicx}

\usepackage{color}
\usepackage{amscd}
\usepackage[top=37truemm,bottom=30truemm,right=25truemm,left=25truemm]{geometry}
\usepackage{pstricks}
\usepackage{enumerate}
\numberwithin{equation}{section}
\newtheorem{theorem}{Theorem}[section]
\newtheorem{proposition}[theorem]{Proposition} 
\newtheorem{lemma}[theorem]{Lemma}
\newtheorem{corollary}[theorem]{Corollary}
\newtheorem{defprop}[theorem]{Definition-Proposition}
\theoremstyle{definition}
\newtheorem{definition}[theorem]{Definition}
\newtheorem{example}[theorem]{Example}

\newcommand{\twosilt}{\mathop{2\text{-}\mathsf{silt}}\nolimits}
\newcommand{\twotilt}{\mathop{2\text{-}\mathsf{tilt}}\nolimits}
\newcommand{\tilt}{\mathop{\mathsf{tilt}}\nolimits}

\newcommand{\soc}{\operatorname{soc}\nolimits}
\newcommand{\Ext}{\operatorname{Ext}\nolimits}

\newcommand{\Hom}{\operatorname{Hom}\nolimits}

\DeclareMathOperator{\thick}{\mathsf{thick}}

\DeclareMathOperator{\moduleCategory}{\mathsf{mod}} 
\renewcommand{\mod}{\moduleCategory}
\DeclareMathOperator{\proj}{\mathsf{proj}}

\newcommand{\Kb}{\mathsf{K}^{\rm b}}
\title[two-term tilting complexes over symmetric algebras with radical cube zero]
{The number of two-term tilting complexes over \\symmetric algebras with radical cube zero}
\author{Takahide Adachi}
\address{T.~Adachi: Faculty of Global and Science Studies, Yamaguchi University, 1677-1 Yoshida, Yamaguchi 753-8541, Japan}
\email{tadachi@yamaguchi-u.ac.jp}
\author{Toshitaka Aoki} 
\address{T.~Aoki: Graduate School of Mathematics, Nagoya University, Frocho, Chikusaku, Nagoya 464-8602, Japan}
\email{m15001d@math.nagoya-u.ac.jp}

\begin{document}

\begin{abstract}
In this paper, we compute the number of two-term tilting complexes for an arbitrary symmetric algebra with radical cube zero over an algebraically closed field.   
Firstly, we give a complete list of symmetric algebras with radical cube zero having only finitely many isomorphism classes of two-term tilting complexes in terms of their associated graphs. 
Secondly, we enumerate the number of two-term tilting complexes for each case in the list.
\end{abstract}

\maketitle

\section{Introduction} 
Tilting theory plays an important role in the study of many areas of mathematics.
A central notion of tilting theory is a tilting complex which is a generalization of a progenerator in Morita theory.
Indeed, its endomorphism algebra is derived equivalent to the original algebra \cite{Rickard89der}.
Hence it is a natural problem to give a classification of tilting complexes for a given algebra.
  
In this paper, we study a classification of two-term tilting complexes for an arbitrary symmetric algebra with radical cube zero over an algebraically closed field $\mathbf{k}$.
Symmetric algebras with radical cube zero have been studied by Okuyama \cite{Okuyama86}, Benson \cite{Benson08} and Erdmann--Solberg \cite{ES}, and also appear in several areas such as \cite{CL,HK,Seidel08}.  
Recently, Green--Schroll \cite{GSa} showed that this class is precisely the Brauer configuration algebras with radical cube zero.

The study of symmetric algebras $A$ with radical cube zero can be reduced to that of algebras with radical square zero. 
For example, as an application of $\tau$-tilting theory (\cite{AIR}), we find in Proposition \ref{reduction} that the functor $-\otimes_A A/\soc A$ gives a bijection
\begin{align}
\twotilt A \longrightarrow \twosilt \, (A/\soc A). \notag
\end{align}
Here, we denote by $\twotilt A$ (respectively, $\twosilt A$) the set of isomorphism classes of basic two-term tilting (respectively two-term silting) complexes for $A$. Notice that tilting complexes coincide with silting complexes for a symmetric algebra $A$ (\cite[Example 2.8]{AI}). 

In \cite{Adachi16b,Aoki18,Zhang13}, they study two-term silting theory (or equivalently $\tau$-tilting theory) for algebras with radical square zero.
The first author (\cite{Adachi16b}) gives a characterization of algebras with radical square zero which are $\tau$-tilting finite (i.e., having only finitely many isomorphism classes of basic two-term silting complexes) by using the notion of single quivers, see Proposition \ref{RSZ}(2). 
Using this result, we give a complete list of $\tau$-tilting finite symmetric algebras with radical cube zero as follows.

Now, let $A$ be a basic connected finite dimensional symmetric $\mathbf{k}$-algebra with radical cube zero.
Let $Q$ be the Gabriel quiver of $A$ and $Q^{\circ}$ the quiver obtained from $Q$ by deleting all loops.
We show in Definition-Proposition \ref{graphA} that $Q^{\circ}$ is the double quiver $Q_G$ (see Definition \ref{def:double quiver}) of a finite connected (undirected) graph $G$ with no loops, i.e., $Q^{\circ}=Q_G$.
We call $G$ the graph of $A$. 

\begin{theorem} \label{theorem1} 
Let $A$ be a basic connected finite dimensional symmetric $\mathbf{k}$-algebra with radical cube zero.  
Then the following conditions are equivalent. 
\begin{enumerate}[\rm (1)]
\item $A$ is $\tau$-tilting finite (or equivalently, $\twotilt A$ is finite).
\item The graph of $A$ is one of graphs in the following list. 
\end{enumerate}

$$
\begin{xy} 
(6, -4)*{(\mathbb{A}_n)}, (-2,-12)*+{1}="A1", (6,-12)*+{2}="A2", (25, -12)*+{n}="An",
(15, -12)*{\cdots}="dot",
{ "A1" \ar@{-} "A2"},
{"A2" \ar@{-} (11,-12)},
{ (19, -12) \ar@{-} "An" },
(48, -4)*{(\mathbb{D}_n)\ 4 \leq n}, (33,-7)*+{1}="D1", (33, -17)*+{2}="D2", (41,-12)*+{3}="D3", 
(60, -12)*+{n}="Dn",
(50, -12)*{\cdots}="dot",
{ "D1" \ar@{-} "D3"},
{ "D2" \ar@{-} "D3"},
{"D3" \ar@{-} (46, -12)},
{ (54, -12) \ar@{-} "Dn"},
(78, -4)*{(\mathbb{E}_6)}, (70, -14)*+{1}="E61", (78, -14)*+{2}="E62", (86, -14)*+{3}="E63", (94, -14)*+{5}="E64", (102, -14)*+{6}="E65", (86, -6)*+{4}="E66", 
{"E61" \ar@{-} "E62"},
{"E62" \ar@{-} "E63"},
{"E63" \ar@{-} "E64"},
{"E64" \ar@{-} "E65"},
{"E63" \ar@{-} "E66"},
(6, -20)*{(\mathbb{E}_7)}, (-2, -31)*+{1}="E71", (6, -31)*+{2}="E72", (14, -31)*+{3}="E73", (22, -31)*+{5}="E74", (30, -31)*+{6}="E75", (38, -31)*+{7}="E76", (14, -23)*+{4}="E77",
{"E71" \ar@{-} "E72"},
{"E72" \ar@{-} "E73"},
{"E73" \ar@{-} "E74"},
{"E74" \ar@{-} "E75"},
{"E75" \ar@{-} "E76"},
{"E73" \ar@{-} "E77"},
(58, -20)*{(\mathbb{E}_8)},
(50, -31)*+{1}="E81", (58, -31)*+{2}="E82", (66, -31)*+{3}="E83", (74, -31)*+{5}="E84", 
(82, -31)*+{6}="E85", (90, -31)*+{7}="E86", (98, -31)*+{8}="E87", (66, -23)*+{4}="E88",
{"E81" \ar@{-} "E82"},
{"E82" \ar@{-} "E83"},
{"E83" \ar@{-} "E84"},
{"E84" \ar@{-} "E85"},
{"E85" \ar@{-} "E86"},
{"E86" \ar@{-} "E87"},
{"E83" \ar@{-} "E88"},
(125, -4)*{(\widetilde{\mathbb{A}}_{n-1}) \ n\colon \mathrm{odd}}, 
(125, -12)*+{1}="b1", (115, -18)*+{2}="b2", (115, -28)*+{3}="b3",(125, -32)*+{4}="b4", 
(135, -18)*+{n}="bn",
(135, -26)*{\rotatebox{90}{$\cdots$}}="ddot",
{"b1" \ar@{-} "b2"},
{"b2" \ar@{-} "b3"},
{"b3" \ar@{-} "b4"},
{"b1" \ar@{-} "bn"},
{"bn" \ar@{-} (135, -22)},
{ (133, -30) \ar@{-} "b4"},
\end{xy}\vspace{3mm} \noindent
$$

$$
\begin{xy}
(26, 0)*{({\rm I}_n) \ 4 \leq n},
(28, -7)*+{1}="c1", (21,-15)*+{2}="c2", (36, -15)*+{3}="c3", 
(36, -23)*+{4}="c4", (36, -28)="c5", (36, -35)="c6", 
(36.5, -31)*{\rotatebox{90}{$\cdots$}}="ddot",
(36, -39)*+{n}="c7",  
{"c1"  \ar@{-} "c2"},
{"c1"  \ar@{-} "c3"},
{"c2"  \ar@{-} "c3"},
{"c3"  \ar@{-} "c4"},
{"c4"  \ar@{-} "c5"},
{"c6"  \ar@{-} "c7"},
(53, 0)*{({\rm II}_n) \ 5 \leq n \leq8},
(52, -7)*+{1}="d1", (45,-15)*+{2}="d2", (59, -15)*+{3}="d3",  
(45, -23)*+{4}="d4", (59, -23)*+{5}="d5", 
(59, -28)="d6", (59, -35)="d7", 
(59.5, -31)*{\rotatebox{90}{$\cdots$}}="edot",
(59, -39)*+{n}="d8",  
{"d1"  \ar@{-} "d2"},
{"d1"  \ar@{-} "d3"},
{"d2"  \ar@{-} "d3"},
{"d2"  \ar@{-} "d4"},
{"d3"  \ar@{-} "d5"},
{"d5"  \ar@{-} "d6"},
{"d7"  \ar@{-} "d8"},
(72, 0)*{({\rm III})},
(79,-7)*+{1}="e1", (79, -15)*+{2}="e2", (72,-22)*+{3}="e3", 
(72, -30)*+{4}="e4", (86,-22)*+{5}="e5", (86,-30)*+{6}="e6",
{"e1"  \ar@{-} "e2"},
{"e2"  \ar@{-} "e3"},
{"e2"  \ar@{-} "e5"},
{"e3"  \ar@{-} "e4"},
{"e4"  \ar@{-} "e6"},
{"e5"  \ar@{-} "e6"},
(97, 0)*{({\rm IV})},
(104, -7)*+{1}="f1",
(104, -15)*+{2}="f2", (97,-23)*+{3}="f3", (111, -23)*+{4}="f4",  
(97, -31)*+{5}="f5", (111, -31)*+{6}="f6",  
{"f1"  \ar@{-} "f2"},
{"f2"  \ar@{-} "f3"},
{"f2"  \ar@{-} "f4"},
{"f3"  \ar@{-} "f4"},
{"f3"  \ar@{-} "f5"},
{"f4"  \ar@{-} "f6"},
(121, 0)*{({\rm V})},
(128, -7)*+{1}="g1", (121,-15)*+{2}="g2", (135, -15)*+{3}="g3",  
(121, -23)*+{4}="g4", (121, -31)*+{5}="g5",  (135, -23)*+{6}="g6", 
(135, -31)*+{7}="g7",
{"g1"  \ar@{-} "g2"},
{"g1"  \ar@{-} "g3"},
{"g2"  \ar@{-} "g3"},
{"g2"  \ar@{-} "g4"},
{"g3"  \ar@{-} "g6"},
{"g4"  \ar@{-} "g5"},
{"g6"  \ar@{-} "g7"}, 
\end{xy}
$$
\end{theorem}

The second author (\cite{Aoki18}) classifies two-term silting complexes for an arbitrary algebra with radical square zero by using tilting modules over a path algebra (see Proposition \ref{RSZ}(1)).
Since the cardinality of the set of isomorphism classes of tilting modules over a path algebra is well known, 
this provides us an explicit way to compute the number of them. We use this result to determine the number $\#\twotilt A$ for each graph $G$ in the list of Theorem \ref{theorem1}. 

\begin{theorem} \label{theorem2}
In Theorem \ref{theorem1}, the number $\#\twotilt A$ depends only on the graph $G$ of $A$ and is given as follows.
 
{\fontsize{9pt}{0.4cm}\selectfont
\begin{table}[h]
{\renewcommand\arraystretch{1.3} 
\begin{tabular}{|c||c|c|c|c|c|c|c|c|c|c|c|c|c|c|} \hline
$G$ & $\mathbb{A}_n$& $\mathbb{D}_{n}$&$\mathbb{E}_6 $&$ \mathbb{E}_7$&$ \mathbb{E}_8$& $\tilde{\mathbb{A}}_{n-1}$ & $ {\rm I}_n$ & 
${\rm II}_5$& ${\rm II}_6$& ${\rm II}_7$ & ${\rm II}_8$ & ${\rm III}$ & ${\rm IV}$ & ${\rm V}$ \\ \hline 
$\# \twotilt A$ &$\binom{2n}{n}$ & $a_n$ & $1700$ & $8872$ & $54066$ & $2^{2n-1}$ & $b_n$ & $632$ & $2936$ & $11306$ & $75240$ & $3108$& $4056$& $17328$   \\ \hline
\end{tabular}
}
\end{table}}
\noindent Here, for any $n\ge 4$, let $a_n:= 6\cdot 4^{n-2}-2\binom{2(n-2)}{n-2}$ and $b_n:=6\cdot 4^{n-2} + 2\binom{2n}{n} -4\binom{2(n-1)}{n-1}-4\binom{2(n-2)}{n-2}$.
\end{theorem}

We remark that the numbers for Dynkin graphs of type $\mathbb{A}$, $\mathbb{D}$ and $\mathbb{E}$ in the list are precisely the biCatalan numbers introduced by \cite{BR} in the context of Coxeter-Catalan combinatorics. 
Our results for Dynkin graphs are independently obtained by \cite{DIRRT} in the study of biCambrian lattices for preprojective algebras.  

We also remark that we can generalize our results for Brauer configuration algebras in terms of multiplicities. 
A Brauer configuration algebra is defined by a configuration and a multiplicity function. 
The configuration of a Brauer configuration algebra with radical cube zero corresponds to a graph \cite{GSa}. 
By \cite{EJR}, one can show that the number of two-term tilting complexes over Brauer configuration algebras is independent of the multiplicity. 
Therefore, we can also apply our results to any Brauer configuration algebra obtained by replacing the multiplicity of a Brauer configuration associated with a graph in the list of Theorem \ref{theorem1}.

This paper is organized as follows.
In Section \ref{sec:preliminaries}, we recall the definition of algebras with radical square zero and their two-term silting theory which are needed in this paper.
In Section \ref{sec:RCZ}, we study symmetric algebras with radical cube zero together with the correspondence algebra with radical square zero.
Our main results are Theorem \ref{reduced ver} and Corollary \ref{number by graph} which provide us an explicit way to compute the number of two-term tilting complexes for a given symmetric algebra with radical cube zero.
In Section \ref{main theorem}, we prove Theorems \ref{theorem1} and \ref{theorem2} by using results shown in the previous section.

\section{Preliminaries} \label{sec:preliminaries}

Throughout this paper, $\mathbf{k}$ is an algebraically closed field. 
We recall that any basic connected finite dimensional $\mathbf{k}$-algebra $A$ is isomorphic to a bound quiver algebra $A\cong \mathbf{k}Q/I$, where $Q$ is a finite connected quiver and $I$ is an admissible ideal in the path algebra $\mathbf{k}Q$ of the quiver $Q$. We call $Q_A:=Q$ the \emph{Gabriel quiver} of $A$.

\subsection{Silting complexes}

Let $A$ be a basic (not necessary connected) finite dimensional $\mathbf{k}$-algebra. We denote by $\mod A$ the category of finitely generated right $A$-modules and by $\proj A$ the category of finitely generated projective right $A$-modules.
Let $\Kb(\proj A)$ denote the homotopy category of bounded complexes of objects of $\proj A$. 
For a complex $X\in \Kb(\proj A)$, we say that $X$ is \emph{basic} if it is a direct sum of pairwise non-isomorphic indecomposable objects. 

\begin{definition}
A complex $T$ in $\Kb(\proj A)$ is said to be \emph{presilting} if it satisfies 
\begin{align}
\Hom_{\Kb(\proj A)}(T,T[i])=0 \notag
\end{align}
for all positive integers $i$.
A presilting complex $T$ is called a \emph{silting complex} if it satisfies $\thick T=\Kb(\proj A)$, where $\thick T$ is the smallest triangulated full subcategory which contains $T$ and is closed under taking direct summands.
In addition, a silting complex $T$ is called a \emph{tilting complex} if $\Hom_{\Kb(\proj A)}(T,T[i])=0$ for all non-zero integers $i$.
\end{definition}

We restrict our interest to the set of two-term silting complexes. Here, a complex $T=(T^{i},d^{i})$ in $\Kb(\proj A)$ is said to be \emph{two-term} if it is isomorphic to a complex concentrated only in degree $0$ and $-1$, i.e., 
\begin{align}
(T^{-1}\overset{d^{-1}}{\rightarrow} T^0) =  \cdots \to 0 \to T^{-1} \overset{d^{-1}}{\longrightarrow} T^0 \to 0 \to \cdots \notag
\end{align}
We denote by $\twosilt A$ (respectively, $\twotilt A$) the set of isomorphic classes of basic two-term silting (respectively, two-term tilting) complexes for $A$. 

Now, we call $M\in \mod A$ a \emph{tilting module} if all the following conditions are satisfied: (i) the projective dimension of $M$ is at most $1$, (ii) $\Ext_A^1(M,M)=0$, and (iii) $|M|=|A|$, where $|M|$ denotes the number of pairwise non-isomorphic indecomposable direct summands of $M$. We denote by $\tilt A$ the set of isomorphism classes of basic tilting $A$-modules. 
By definition, we can naturally regard a tilting $A$-module $M$ as a tilting complex. More precisely, by taking a minimal projective presentation $P_1 \overset{f}{\to} P_0\to M \to 0$ of $M$ in $\mod A$, the two-term complex $(P_{1} \overset{f}{\to} P_0)$ provides a tilting complex in $\Kb(\proj A)$. 

The number of tilting modules over a path algebra of a Dynkin quiver is well known. 

\begin{proposition}{\rm (see \cite{ONFR} for example)} \label{tilting number}
Let $Q$ be a quiver whose underlying graph $\Delta$ is one of Dynkin graphs of type $\mathbb{A}$, $\mathbb{D}$ and $\mathbb{E}$. Then the number $\#\tilt \mathbf{k}Q$ is given by the following table and does not depend on the orientation of $Q$.
\begin{table}[h] 
\begin{center}
{\renewcommand\arraystretch{1.3} 
\begin{tabular}{|c||c|c|c|c|c|c|} \hline
$\Delta$   &  $\mathbb{A}_n \, (n\geq 1)$  &$\ \mathbb{D}_n \,(n\geq4)$& $\mathbb{E}_6$ &  
$\mathbb{E}_7$ & $\mathbb{E}_8$  \\ \hline
$\# \tilt \mathbf{k}Q$ & $\frac{1}{n+1}\binom{2n}{n}$ & $\frac{3n-4}{2n}\binom{2(n-1)}{n-1}$ & $418$ & $2431$ & $17342$\\ \hline
\end{tabular}}
\end{center} 
\end{table}
\end{proposition} 

More generally, if $Q$ is a disjoint union of Dynkin quivers $Q_{\lambda}$ ($\lambda \in \Lambda$), then we have 
\begin{equation} \label{disjoint Dynkin}
  \#\tilt \mathbf{k}Q = \prod_{\lambda\in \Lambda} \#\tilt \mathbf{k}Q_{\lambda}
\end{equation}
and this number is completely determined by a collection of the underlying graphs $\Delta_{\lambda}$ of $Q_{\lambda}$ for all $\lambda\in \Lambda$ as in Proposition \ref{tilting number}.

\subsection{Algebras with radical square zero}

Let $A$ be a basic connected finite dimensional $\mathbf{k}$-algebra. 
We say that $A$ is an algebra with \emph{radical square zero} (respectively, \emph{radical cube zero}) if $J^2=0$ but $J\neq 0$ (respectively, $J^3=0$ but $J^2\neq 0$), where $J$ is the Jacobson radical of $A$. 
For simplicity, we abbreviate an algebra with radical square zero (respectively, radical cube zero) by a RSZ (respectively, RCZ) algebra.

We first recall that any basic connected finite dimensional RSZ $\mathbf{k}$-algebra $A$ is isomorphic to a bound quiver algebra $\mathbf{k}Q/I$, where $Q:=Q_A$ is the Gabriel quiver of $A$ and $I$ is the two-sided ideal in $\mathbf{k}Q$ generated by all paths of length $2$.

Next, let $Q=(Q_{0},Q_{1})$ be a finite connected quiver, where $Q_{0}$ is the vertex set and $Q_{1}$ is the arrow set. 
We denote by $Q^{\rm op}$ the opposite quiver of $Q$. 
For a map $\epsilon\colon Q_{0}\to \{ \pm 1\}$, we define a quiver $Q_{\epsilon}$, called a {\it single quiver} of $Q$, as follows: 
\begin{itemize}
\item The set of vertices is $Q_0$. 
\item We draw an arrow $a \colon i\to j$ in $Q_{\epsilon}$ whenever there exists an arrow $a\colon i\to j$ with $\epsilon(i)=+1$ and $\epsilon(j)=-1$. 
\end{itemize}
Note that $Q_{\epsilon}$ is bipartite (i.e., each vertex is either a sink or a source), but not connected in general.
Since it has no loops by definition, we have $Q_{\epsilon}=(Q^{\circ})_{\epsilon}$, where $Q^{\circ}$ denotes the quiver obtained from $Q$ by deleting all loops. 

We give a connection between two-term silting complexes for a RSZ algebra and tilting modules over path algebras.

\begin{proposition}\label{RSZ}
Let $A$ be a basic connected finite dimensional RSZ $\mathbf{k}$-algebra and $Q_A$ the Gabriel quiver of $A$.
Let $Q:=(Q_A)^{\circ}$ be the quiver obtained from $Q_A$ by deleting all loops.
Then the following statements hold.
\begin{enumerate}[\rm (1)]
\item \textnormal{(\cite[Theorem 1.1]{Aoki18})} There is a bijection
\begin{align}
\twosilt A \longrightarrow \bigsqcup_{\epsilon \colon Q_{0} \rightarrow\{\pm 1\}} 
\tilt \mathbf{k}(Q_{\epsilon})^{\rm op}. \notag
\end{align}
\item \textnormal{(\cite{Adachi16b,Aoki18})} The following conditions are equivalent.
\begin{enumerate}[\rm (a)]
\item $\twosilt A$ is finite.
\item For every map $\epsilon\colon Q_{0}\to \{ \pm 1\}$, the underlying graph of the single quiver $Q_{\epsilon}$ is a disjoint union of Dynkin graphs of type $\mathbb{A}$, $\mathbb{D}$ and $\mathbb{E}$.
\end{enumerate}
\item If one of equivalent conditions of \textnormal{(2)} holds, we have 
\begin{align}
\#\twosilt A = \sum_{\epsilon \colon Q_0 \to \{\pm1\}} \# \tilt \mathbf{k}(Q_{\epsilon})^{\rm op}.\notag
\end{align}
\end{enumerate}
\end{proposition}

We remark that we can replace the quiver $Q$ with the Gabriel quiver $Q_A$ of $A$ in Proposition \ref{RSZ} since we have $(Q_A)_{\epsilon} = Q_{\epsilon}$ for any map $\epsilon\colon Q_0 \to \{\pm1\}$.

\section{Two-term tilting complexes over symmetric RCZ algebras} \label{sec:RCZ}

Let $A$ be a basic connected finite dimensional symmetric RCZ $\mathbf{k}$-algebra. Then $\overline{A}:=A/\soc A$ is a RSZ algebra by definition. Moreover, the Gabriel quiver of $\overline{A}$ coincides with the Gabriel quiver of $A$ since $\soc A$ is contained in the square of the Jacobson radical of $A$. 

The following is basic. Here, we remember that silting complexes coincide with tilting complexes for a symmetric algebra $A$ (\cite[Example 2.8]{AI}). In particular, $\twotilt A=\twosilt A$. 

\begin{proposition} \cite[Theorem 3.3]{Adachi16a} \label{reduction}
Let $A$ be a basic connected finite dimensional symmetric RCZ $\mathbf{k}$-algebra and $\overline{A}:=A/\soc A$. Then the functor $-\otimes_A \overline{A}$ gives a bijection 
\begin{align}
\twotilt A \longrightarrow \twosilt \overline{A}. \notag
\end{align}  
\end{proposition}

Next, the following observations provide us a combinatorial framework of studying two-term tilting complexes over symmetric RCZ algebras. 

\begin{definition} \label{def:double quiver}
For a finite connected graph $G$ with no loops, we define a quiver $Q_G$ as follows. 
\begin{itemize}
\item The set of vertices of $Q_G$ is the set of vertices of $G$.  
\item We draw two arrows $a^{\ast} \colon i\to j$ and $a^{\ast\ast} \colon j\to i$ whenever there exists an edge $a$ of $G$ connecting $i$ and $j$.  
\end{itemize}
We call $Q_G$ the \emph{double quiver} of $G$. Notice that $Q_G$ has no loops since so does $G$.
\end{definition}

\begin{defprop} \label{graphA}
Let $A$ be a basic connected finite dimensional symmetric RCZ $\mathbf{k}$-algebra. 
Let $Q_A$ be the Gabriel quiver of $A$ and $Q:=(Q_A)^{\circ}$ the quiver obtained from $Q_A$ by deleting all loops. Then $Q$ is the double quiver $Q_G$ of a finite connected (undirected) graph $G$ with no loops. We call $G$ the graph of $A$. 
\end{defprop}

\begin{proof}
For the Gabriel quiver $Q_A$ of $A$, let $\pi \colon \mathbf{k}Q_A \to A$ be a canonical surjection. 
For any vertex $i$ of $Q_A$, let $P_i$ be the indecomposable projective $A$-module corresponding to $i$. 
By definition, $P_i$ has Loewy length $3$ and its simple socle is isomorphic to the simple top $S_{i}:=P_{i}/P_{i}J$.

We recall from \cite[Proposition 5.6]{GSa} that our algebra $A$ is special multiserial (we refer to \cite[Definition 2.2]{GSa} for the definition of special multiserial algebras). Then each arrow $a\colon i\to j$ of $Q_A$ determines the unique arrow $\sigma(a)$ such that $\pi(a\sigma(a))\neq 0$, and the correspondence $\sigma$ gives a permutation of the set of arrows of $Q_A$, see \cite[Definition 4.8]{GSa}. In addition, the element $\pi(a\sigma(a)\sigma^2(a)\cdots\sigma^{m-1}(a))$ lies in the socle of $P_i$, where $m$ is the length of the $\sigma$-orbit containing the arrow $a$. 
Since $P_i$ has Loewy length $3$, $m=2$ must hold. In particular, $\sigma(a)$ is the unique arrow $\sigma(a)\colon j\to i$ such that $\pi(\sigma(a)a)\neq 0$. 

Now, we can restrict the permutation $\sigma$ to the subset consisting of all arrows which are not loops. Then we define a finite undirected graph $G$ as follows: The set of vertices of $G$ bijectively corresponds to the set of vertices of $Q_A$, and the set of edges of $G$ is naturally given by the set of unordered pairs $\{a,\sigma(a)\}$ for all arrows $a$ of $Q_A$ which are not loops. Then $G$ is the desired one as $(Q_A)^{\circ}=Q_G$ from our construction.
\end{proof}

As we mentioned before, the algebras $A$ and $\overline{A}:=A/\soc A$ have the same Gabriel quiver $Q_A = Q_{\overline{A}}$. Therefore, $(Q_A)^{\circ}= (Q_{\overline{A}})^{\circ}$ is the double quiver $Q_G$ of a common finite connected graph $G$ with no loops by Definition-Proposition \ref{graphA}. 

\begin{theorem} \label{reduced ver}
Let $A$ be a basic connected finite dimensional symmetric RCZ $\mathbf{k}$-algebra and $\overline{A}:=A/\soc A$.  
Let $Q_A$ be the Gabriel quiver of $A$ and $Q:=(Q_A)^{\circ}$ the quiver obtained from $Q_A$ by deleting all loops.
\begin{enumerate}[\rm (1)]
\item The following conditions are equivalent.
\begin{enumerate}[\rm (a)]
\item $\twotilt A$ is finite.
\item $\twosilt \overline{A}$ is finite.  
\item For every map $\epsilon \colon Q_{0} \rightarrow \{\pm 1\}$, the underlying graph of the single quiver $Q_{\epsilon}$ is a disjoint union of Dynkin graphs of type $\mathbb{A}$, $\mathbb{D}$ and $\mathbb{E}$. 
\end{enumerate}
\item Fix any vertex $v\in Q_{0}$. If one of the equivalent conditions in \textnormal{(1)} is satisfied, then the following equalities hold.
\begin{align}
\# \twotilt A = \# \twosilt \overline{A} = 2 \cdot \sum_{\substack{\epsilon\colon Q_{0} \rightarrow \{\pm 1\} \\ \epsilon(v)=+1}} \# \tilt \mathbf{k}{Q}_{\epsilon}. \notag 
\end{align}
\end{enumerate}
\end{theorem}
\begin{proof}
(1) It follows from Propositions \ref{RSZ}(2) and \ref{reduction}.
  
(2) By Proposition \ref{reduction}, we have $\# \twotilt A = \# \twosilt \overline{A}$.
We show the second equality.
Let $v$ be a vertex in $Q$.
By Proposition \ref{RSZ}(1), we have 
\begin{align}
\# \twosilt \overline{A}
= \sum_{\epsilon \colon Q_{0} \rightarrow\{\pm 1\}} \#\tilt \mathbf{k}(Q_{\epsilon})^{\rm op}
= \sum_{\substack{\epsilon\colon Q_{0} \rightarrow \{\pm 1\} \\ \epsilon(v)=+1}} \# \tilt \mathbf{k}(Q_{\epsilon})^{\rm op}
+ \sum_{\substack{\epsilon\colon Q_{0} \rightarrow \{\pm 1\} \\ \epsilon(v)=-1}} \# \tilt \mathbf{k}(Q_{\epsilon})^{\rm op}. \notag
\end{align}
For a map $\epsilon\colon Q_{0}\to \{ \pm 1\}$, we define a map $-\epsilon\colon Q_{0}\to \{ \pm 1\}$ by $(-\epsilon)(i):=-\epsilon(i)$ for all $i\in Q_{0}$.
Since $Q$ is the double quiver of the graph $G$ of $A$, we have $Q_{-\epsilon}=(Q_{\epsilon})^{\mathrm{op}}$. 
This implies that $Q_{\epsilon}$ and $Q_{-\epsilon}$ have the same underlying graph $\Delta$.
By our assumption, $\Delta$ is a disjoint union of Dynkin graphs.
Thus we obtain $\# \tilt \mathbf{k}Q_{\epsilon}= \# \tilt \mathbf{k}Q_{-\epsilon}$ because the number of non-isomorphic tilting modules over a path algebra of Dynkin type does not depend on orientation, see Proposition \ref{tilting number}.
Hence we have 
\begin{align}
\sum_{\substack{\epsilon\colon Q_{0} \rightarrow \{\pm 1\} \\ \epsilon(v)=+1}} \# \tilt \mathbf{k}Q_{\epsilon}
=\sum_{\substack{\epsilon\colon Q_{0} \rightarrow \{\pm 1\} \\ \epsilon(v)=+1}} \# \tilt \mathbf{k}(Q_{\epsilon})^{\rm op}
= \sum_{\substack{\epsilon\colon Q_{0} \rightarrow \{\pm 1\} \\ \epsilon(v)=-1}} \# \tilt \mathbf{k}(Q_{\epsilon})^{\rm op}. \notag
\end{align}
This finishes the proof. 
\end{proof}

For our convenience, we restate Theorem \ref{reduced ver} in terms of undirected graphs. 
Let $G=(G_0,G_1)$ be a finite connected graph with no loops, where $G_0$ is the set of vertices and $G_1$ is the set of edges. For each map $\epsilon \colon G_0\to \{\pm1\}$, let $G_{\epsilon}$ be the graph obtained from $G$ by removing all edges between vertices $i,j$ with $\epsilon(i)=\epsilon(j)$. From our construction, $G_{\epsilon}$ is precisely the underlying graph of the quiver $Q_{\epsilon}$, where $Q:=Q_G$ is the double quiver of $G$ with vertex set $Q_0=G_0$. In particular, $Q_{\epsilon}$ is a disjoint union of Dynkin quivers if and only if $G_{\epsilon}$ is a disjoint union of Dynkin graphs. 

Now, we recall that, for a quiver $Q$ whose underlying graph $\Delta$ is a disjoint union of Dynkin graphs, the number $\#\tilt \mathbf{k}Q$ does not depend on orientation of $Q$ and given by (\ref{disjoint Dynkin}). Then, we set $|\Delta| := \# \tilt \mathbf{k}Q$. 

\begin{corollary} \label{number by graph}
Let $A$ be a basic connected finite dimensional symmetric RCZ $\mathbf{k}$-algebra and $G$ the graph of $A$. 
\begin{enumerate}[\rm (1)]
\item The following conditions are equivalent. 
\begin{enumerate}[\rm (a)]
\item $\twotilt A$ is finite. 
\item For every map $\epsilon \colon G_0 \to \{\pm1\}$, the graph $G_{\epsilon}$ is a disjoint union of Dynkin graphs of type $\mathbb{A}$, $\mathbb{D}$ and $\mathbb{E}$. 
\end{enumerate}
\item Assume that, for any $\epsilon \colon G_0\to \{\pm1\}$, the graph $G_{\epsilon}$ is a disjoint union of Dynkin graphs $\Delta_{\epsilon,\lambda}$ ($\lambda \in \Lambda_{\epsilon}$). Then for a fixed vertex $v$ of $G$, the number $\#\twotilt A$ is equal to 
\end{enumerate}
\begin{align}\label{number G}
2\cdot \sum_{\substack{\epsilon\colon G_0 \to \{\pm1\} \\ \epsilon(v)=+1}} |G_{\epsilon}| \ = \ 
2\cdot \sum_{\substack{\epsilon\colon G_0 \to \{\pm1\} \\ \epsilon(v)=+1}} \prod_{\lambda\in \Lambda_{\epsilon}} |\Delta_{\epsilon,\lambda}|. 
\end{align}
\end{corollary}
\begin{proof}
Let $Q:=(Q_A)^{\circ}$, where $Q_A$ is the Gabriel quiver of $A$. Then $Q=Q_G$ holds by Definition-Proposition \ref{graphA}. 
Then the assertion follows from Theorem \ref{reduced ver} since $G_{\epsilon}=Q_{\epsilon}$ for any map $\epsilon \colon G_0\to \{\pm1\}$.
\end{proof}

\begin{definition} \label{number GG}
Keeping the notations in Corollary \ref{number by graph}(2), we write $||G||$ for the number given by the left hand side of (\ref{number G}). 
\end{definition}

\begin{figure}[t]
\begin{center}
\scalebox{1}{
$$
\begin{xy}
(-17,0)*{Q:=Q_{\mathbb{E}_6}\colon}, (0, -1)*+{1}="E61", (14, -1)*+{2}="E62", (28, -1)*+{3}="E63", (42, -1)*+{5}="E64", (56, -1)*+{6}="E65", (28, 10)*+{4}="E66", 
{"E61" \ar@{->}@<1mm> "E62"},
{"E62" \ar@{->}@<1mm> "E63"},
{"E63" \ar@{->}@<1mm> "E64"},
{"E64" \ar@{->}@<1mm> "E65"},
{"E63" \ar@{->}@<1mm> "E66"},
{"E61" \ar@{<-}@<-1mm> "E62"},
{"E62" \ar@{<-}@<-1mm> "E63"},
{"E63" \ar@{<-}@<-1mm> "E64"},
{"E64" \ar@{<-}@<-1mm> "E65"},
{"E63" \ar@{<-}@<-1mm> "E66"},
\end{xy}
$$ 
}
\end{center}
\begin{center}
\scalebox{0.87}{
\vspace{5mm}
\begin{tabular}{|c||c|c|c|c|c|c|} \hline
$Q_{\epsilon}$   &
\scalebox{0.7}{$\begin{xy}
(0, 0)*+{1^+}="E61", (10, 0)*+{2^+}="E62", (20, 0)*+{3^+}="E63", (30, 0)*+{5^+}="E64", (40, 0)*+{6^+}="E65", (20, 10)*+{4^+}="E66", 
(20, -8)*+{(\mathbb{A}_1,\mathbb{A}_1,\mathbb{A}_1,\mathbb{A}_1,\mathbb{A}_1,\mathbb{A}_1)},
\end{xy}$}
&
\scalebox{0.7}{$\begin{xy}
(0, 0)*+{1^-}="E61", (10, 0)*+{2^+}="E62", (20, 0)*+{3^+}="E63", (30, 0)*+{5^+}="E64", (40, 0)*+{6^+}="E65", (20, 10)*+{4^+}="E66", 
{"E61" \ar@{<-} "E62"},
(20, -8)*+{(\mathbb{A}_2,\mathbb{A}_1,\mathbb{A}_1,\mathbb{A}_1,\mathbb{A}_1)},
\end{xy}$}
& 
\scalebox{0.7}{$\begin{xy}
(0, 0)*+{1^+}="E61", (10, 0)*+{2^-}="E62", (20, 0)*+{3^+}="E63", (30, 0)*+{5^+}="E64", (40, 0)*+{6^+}="E65", (20, 10)*+{4^+}="E66", 
{"E61" \ar@{->} "E62"},
{"E62" \ar@{<-} "E63"},
(20, -8)*+{(\mathbb{A}_3, \mathbb{A}_1,\mathbb{A}_1,\mathbb{A}_1)},
\end{xy}$}
&  
\scalebox{0.7}{$\begin{xy}
(0, 0)*+{1^-}="E61", (10, 0)*+{2^-}="E62", (20, 0)*+{3^+}="E63", (30, 0)*+{5^+}="E64", (40, 0)*+{6^+}="E65", (20, 10)*+{4^+}="E66", 
{"E62" \ar@{<-} "E63"},
(20, -8)*+{(\mathbb{A}_1,\mathbb{A}_2, \mathbb{A}_1,\mathbb{A}_1,\mathbb{A}_1)},
\end{xy}$}
\\ \hline
$\# \tilt \mathbf{k}Q_{\epsilon}$ & $1$& $2$ & $5$ & $2$  \\ 
\hline \hline 
&
\scalebox{0.7}{$\begin{xy}
(0, 0)*+{1^+}="E61", (10, 0)*+{2^+}="E62", (20, 0)*+{3^-}="E63", (30, 0)*+{5^+}="E64", (40, 0)*+{6^+}="E65", (20, 10)*+{4^+}="E66", 
{"E62" \ar@{->} "E63"},
{"E63" \ar@{<-} "E64"},
{"E63" \ar@{<-} "E66"},
(20, -8)*+{(\mathbb{A}_1,\mathbb{D}_4, \mathbb{A}_1)},
\end{xy}$}
&
\scalebox{0.7}{$\begin{xy}
(0, 0)*+{1^-}="E61", (10, 0)*+{2^+}="E62", (20, 0)*+{3^-}="E63", (30, 0)*+{5^+}="E64", (40, 0)*+{6^+}="E65", (20, 10)*+{4^+}="E66", 
{"E61" \ar@{<-} "E62"},
{"E62" \ar@{->} "E63"},
{"E63" \ar@{<-} "E64"},
{"E63" \ar@{<-} "E66"},
(20, -8)*+{(\mathbb{D}_5,\mathbb{A}_1)},
\end{xy}$}
& 
\scalebox{0.7}{$\begin{xy}
(0, 0)*+{1^+}="E61", (10, 0)*+{2^-}="E62", (20, 0)*+{3^-}="E63", (30, 0)*+{5^+}="E64", (40, 0)*+{6^+}="E65", (20, 10)*+{4^+}="E66", 
{"E61" \ar@{->} "E62"},
{"E63" \ar@{<-} "E64"},
{"E63" \ar@{<-} "E66"},
(20, -8)*+{(\mathbb{A}_2,\mathbb{A}_3,\mathbb{A}_1)},
\end{xy}$}
&  
\scalebox{0.7}{$\begin{xy}
(0, 0)*+{1^-}="E61", (10, 0)*+{2^-}="E62", (20, 0)*+{3^-}="E63", (30, 0)*+{5^+}="E64", (40, 0)*+{6^+}="E65", (20, 10)*+{4^+}="E66", 
{"E63" \ar@{<-} "E64"},
{"E63" \ar@{<-} "E66"},
(20, -8)*+{(\mathbb{A}_1,\mathbb{A}_1,\mathbb{A}_3,\mathbb{A}_1)},
\end{xy}$}
\\ \hline
& $20$ & $77$ & $10$ & $5$  \\ 
\hline \hline
&
\scalebox{0.7}{$\begin{xy}
(0, 0)*+{1^+}="E61", (10, 0)*+{2^+}="E62", (20, 0)*+{3^+}="E63", (30, 0)*+{5^+}="E64", (40, 0)*+{6^+}="E65", (20, 10)*+{4^-}="E66", 
{"E63" \ar@{->} "E66"},
(20, -8)*+{(\mathbb{A}_1,\mathbb{A}_1,\mathbb{A}_2, \mathbb{A}_1,\mathbb{A}_1)},
\end{xy}$}
&
\scalebox{0.7}{$\begin{xy}
(0, 0)*+{1^-}="E61", (10, 0)*+{2^+}="E62", (20, 0)*+{3^+}="E63", (30, 0)*+{5^+}="E64", (40, 0)*+{6^+}="E65", (20, 10)*+{4^-}="E66", 
{"E61" \ar@{<-} "E62"},
{"E63" \ar@{->} "E66"},
(20, -8)*+{(\mathbb{A}_2,\mathbb{A}_2,\mathbb{A}_1,\mathbb{A}_1)},
\end{xy}$}
& 
\scalebox{0.7}{$\begin{xy}
(0, 0)*+{1^+}="E61", (10, 0)*+{2^-}="E62", (20, 0)*+{3^+}="E63", (30, 0)*+{5^+}="E64", (40, 0)*+{6^+}="E65", (20, 10)*+{4^-}="E66", 
{"E61" \ar@{->} "E62"},
{"E62" \ar@{<-} "E63"},
{"E63" \ar@{->} "E66"},
(20, -8)*+{(\mathbb{A}_4, \mathbb{A}_1,\mathbb{A}_1)},
\end{xy}$}
&  
\scalebox{0.7}{$\begin{xy}
(0, 0)*+{1^-}="E61", (10, 0)*+{2^-}="E62", (20, 0)*+{3^+}="E63", (30, 0)*+{5^+}="E64", (40, 0)*+{6^+}="E65", (20, 10)*+{4^-}="E66", 
{"E62" \ar@{<-} "E63"},
{"E63" \ar@{->} "E66"},
(20, -8)*+{(\mathbb{A}_1,\mathbb{A}_3,\mathbb{A}_1,\mathbb{A}_1)},
\end{xy}$}
\\ \hline
& $2$ & $4$ & $14$ & $5$  \\
\hline \hline
&
\scalebox{0.7}{$\begin{xy}
(0, 0)*+{1^+}="E61", (10, 0)*+{2^+}="E62", (20, 0)*+{3^-}="E63", (30, 0)*+{5^+}="E64", (40, 0)*+{6^+}="E65", (20, 10)*+{4^-}="E66", 
{"E62" \ar@{->} "E63"},
{"E63" \ar@{<-} "E64"},
(20, -8)*+{(\mathbb{A}_1,\mathbb{A}_3,\mathbb{A}_1,\mathbb{A}_1)},
\end{xy}$}
&
\scalebox{0.7}{$\begin{xy}
(0, 0)*+{1^-}="E61", (10, 0)*+{2^+}="E62", (20, 0)*+{3^-}="E63", (30, 0)*+{5^+}="E64", (40, 0)*+{6^+}="E65", (20, 10)*+{4^-}="E66", 
{"E61" \ar@{<-} "E62"},
{"E62" \ar@{->} "E63"},
{"E63" \ar@{<-} "E64"},
(20, -8)*+{(\mathbb{A}_4, \mathbb{A}_1,\mathbb{A}_1)},
\end{xy}$}
& 
\scalebox{0.7}{$\begin{xy}
(0, 0)*+{1^+}="E61", (10, 0)*+{2^-}="E62", (20, 0)*+{3^-}="E63", (30, 0)*+{5^+}="E64", (40, 0)*+{6^+}="E65", (20, 10)*+{4^-}="E66", 
{"E61" \ar@{->} "E62"},
{"E63" \ar@{<-} "E64"},
(20, -8)*+{(\mathbb{A}_2,\mathbb{A}_2,\mathbb{A}_1,\mathbb{A}_1)},
\end{xy}$}
&  
\scalebox{0.7}{$\begin{xy}
(0, 0)*+{1^-}="E61", (10, 0)*+{2^-}="E62", (20, 0)*+{3^-}="E63", (30, 0)*+{5^+}="E64", (40, 0)*+{6^+}="E65", (20, 10)*+{4^-}="E66", 
{"E63" \ar@{<-} "E64"},
(20, -8)*+{(\mathbb{A}_1,\mathbb{A}_1,\mathbb{A}_2,\mathbb{A}_1,\mathbb{A}_1)},
\end{xy}$}
\\ \hline
& $5$ & $14$ & $4$ & $2$  \\
\hline \hline
&
\scalebox{0.7}{$\begin{xy}
(0, 0)*+{1^+}="E61", (10, 0)*+{2^+}="E62", (20, 0)*+{3^+}="E63", (30, 0)*+{5^-}="E64", (40, 0)*+{6^+}="E65", (20, 10)*+{4^+}="E66", 
{"E63" \ar@{->} "E64"},
{"E64" \ar@{<-} "E65"},
(20, -8)*+{(\mathbb{A}_1,\mathbb{A}_1,\mathbb{A}_1,\mathbb{A}_3)},
\end{xy}$}
&
\scalebox{0.7}{$\begin{xy}
(0, 0)*+{1^-}="E61", (10, 0)*+{2^+}="E62", (20, 0)*+{3^+}="E63", (30, 0)*+{5^-}="E64", (40, 0)*+{6^+}="E65", (20, 10)*+{4^+}="E66", 
{"E61" \ar@{<-} "E62"},
{"E63" \ar@{->} "E64"},
{"E64" \ar@{<-} "E65"},
(20, -8)*+{(\mathbb{A}_2,\mathbb{A}_1,\mathbb{A}_3)},
\end{xy}$}
& 
\scalebox{0.7}{$\begin{xy}
(0, 0)*+{1^+}="E61", (10, 0)*+{2^-}="E62", (20, 0)*+{3^+}="E63", (30, 0)*+{5^-}="E64", (40, 0)*+{6^+}="E65", (20, 10)*+{4^+}="E66", 
{"E61" \ar@{->} "E62"},
{"E62" \ar@{<-} "E63"},
{"E63" \ar@{->} "E64"},
{"E64" \ar@{<-} "E65"},
(20, -8)*+{(\mathbb{A}_5,\mathbb{A}_1)},
\end{xy}$}
&  
\scalebox{0.7}{$\begin{xy}
(0, 0)*+{1^-}="E61", (10, 0)*+{2^-}="E62", (20, 0)*+{3^+}="E63", (30, 0)*+{5^-}="E64", (40, 0)*+{6^+}="E65", (20, 10)*+{4^+}="E66", 
{"E62" \ar@{<-} "E63"},
{"E63" \ar@{->} "E64"},
{"E64" \ar@{<-} "E65"},
(20, -8)*+{(\mathbb{A}_1,\mathbb{A}_1, \mathbb{A}_4)},
\end{xy}$}
\\ \hline
& $5$& $10$ & $42$ & $14$  \\ 
\hline \hline 
&
\scalebox{0.7}{$\begin{xy}
(0, 0)*+{1^+}="E61", (10, 0)*+{2^+}="E62", (20, 0)*+{3^-}="E63", (30, 0)*+{5^-}="E64", (40, 0)*+{6^+}="E65", (20, 10)*+{4^+}="E66", 
{"E62" \ar@{->} "E63"},
{"E64" \ar@{<-} "E65"},
{"E63" \ar@{<-} "E66"},
(20, -8)*+{(\mathbb{A}_1\mathbb{A}_3, \mathbb{A}_2)},
\end{xy}$}
&
\scalebox{0.7}{$\begin{xy}
(0, 0)*+{1^-}="E61", (10, 0)*+{2^+}="E62", (20, 0)*+{3^-}="E63", (30, 0)*+{5^-}="E64", (40, 0)*+{6^+}="E65", (20, 10)*+{4^+}="E66", 
{"E61" \ar@{<-} "E62"},
{"E62" \ar@{->} "E63"},
{"E64" \ar@{<-} "E65"},
{"E63" \ar@{<-} "E66"},
(20, -8)*+{(\mathbb{A}_4, \mathbb{A}_2)},
\end{xy}$}
& 
\scalebox{0.7}{$\begin{xy}
(0, 0)*+{1^+}="E61", (10, 0)*+{2^-}="E62", (20, 0)*+{3^-}="E63", (30, 0)*+{5^-}="E64", (40, 0)*+{6^+}="E65", (20, 10)*+{4^+}="E66", 
{"E61" \ar@{->} "E62"},
{"E64" \ar@{<-} "E65"},
{"E63" \ar@{<-} "E66"},
(20, -8)*+{(\mathbb{A}_2,\mathbb{A}_2, \mathbb{A}_2)},
\end{xy}$}
&  
\scalebox{0.7}{$\begin{xy}
(0, 0)*+{1^-}="E61", (10, 0)*+{2^-}="E62", (20, 0)*+{3^-}="E63", (30, 0)*+{5^-}="E64", (40, 0)*+{6^+}="E65", (20, 10)*+{4^+}="E66", 
{"E64" \ar@{<-} "E65"},
{"E63" \ar@{<-} "E66"},
(20, -8)*+{(\mathbb{A}_1, \mathbb{A}_1, \mathbb{A}_2,\mathbb{A}_2)},
\end{xy}$}
\\ \hline
& $10$ & $28$ & $8$ & $4$  \\
\hline \hline
&
\scalebox{0.7}{$\begin{xy}
(0, 0)*+{1^+}="E61", (10, 0)*+{2^+}="E62", (20, 0)*+{3^+}="E63", (30, 0)*+{5^-}="E64", (40, 0)*+{6^+}="E65", (20, 10)*+{4^-}="E66", 
{"E63" \ar@{->} "E64"},
{"E64" \ar@{<-} "E65"},
{"E63" \ar@{->} "E66"},
(20, -8)*+{(\mathbb{A}_1,\mathbb{A}_1, \mathbb{A}_4)},
\end{xy}$}
&
\scalebox{0.7}{$\begin{xy}
(0, 0)*+{1^-}="E61", (10, 0)*+{2^+}="E62", (20, 0)*+{3^+}="E63", (30, 0)*+{5^-}="E64", (40, 0)*+{6^+}="E65", (20, 10)*+{4^-}="E66", 
{"E61" \ar@{<-} "E62"},
{"E63" \ar@{->} "E64"},
{"E64" \ar@{<-} "E65"},
{"E63" \ar@{->} "E66"},
(20, -8)*+{(\mathbb{A}_2, \mathbb{A}_4)},
\end{xy}$}
& 
\scalebox{0.7}{$\begin{xy}
(0, 0)*+{1^+}="E61", (10, 0)*+{2^-}="E62", (20, 0)*+{3^+}="E63", (30, 0)*+{5^-}="E64", (40, 0)*+{6^+}="E65", (20, 10)*+{4^-}="E66", 
{"E61" \ar@{->} "E62"},
{"E62" \ar@{<-} "E63"},
{"E63" \ar@{->} "E64"},
{"E64" \ar@{<-} "E65"},
{"E63" \ar@{->} "E66"},
(20, -8)*+{(\mathbb{E}_6)},
\end{xy}$}
&  
\scalebox{0.7}{$\begin{xy}
(0, 0)*+{1^-}="E61", (10, 0)*+{2^-}="E62", (20, 0)*+{3^+}="E63", (30, 0)*+{5^-}="E64", (40, 0)*+{6^+}="E65", (20, 10)*+{4^-}="E66", 
{"E62" \ar@{<-} "E63"},
{"E63" \ar@{->} "E64"},
{"E64" \ar@{<-} "E65"},
{"E63" \ar@{->} "E66"},
(20, -8)*+{(\mathbb{A}_1, \mathbb{D}_5)},
\end{xy}$}
\\ \hline
& $14$ & $28$ & $418$ & $77$  \\
\hline \hline
&
\scalebox{0.7}{$\begin{xy}
(0, 0)*+{1^+}="E61", (10, 0)*+{2^+}="E62", (20, 0)*+{3^-}="E63", (30, 0)*+{5^-}="E64", (40, 0)*+{6^+}="E65", (20, 10)*+{4^-}="E66", 
{"E62" \ar@{->} "E63"},
{"E64" \ar@{<-} "E65"},
(20, -8)*+{(\mathbb{A}_1, \mathbb{A}_2, \mathbb{A}_1, \mathbb{A}_2)},
\end{xy}$}
&
\scalebox{0.7}{$\begin{xy}
(0, 0)*+{1^-}="E61", (10, 0)*+{2^+}="E62", (20, 0)*+{3^-}="E63", (30, 0)*+{5^-}="E64", (40, 0)*+{6^+}="E65", (20, 10)*+{4^-}="E66", 
{"E61" \ar@{<-} "E62"},
{"E62" \ar@{->} "E63"},
{"E64" \ar@{<-} "E65"},
(20, -8)*+{(\mathbb{A}_3, \mathbb{A}_1, \mathbb{A}_2)},
\end{xy}$}
& 
\scalebox{0.7}{$\begin{xy}
(0, 0)*+{1^+}="E61", (10, 0)*+{2^-}="E62", (20, 0)*+{3^-}="E63", (30, 0)*+{5^-}="E64", (40, 0)*+{6^+}="E65", (20, 10)*+{4^-}="E66", 
{"E61" \ar@{->} "E62"},
{"E64" \ar@{<-} "E65"},
(20, -8)*+{(\mathbb{A}_2, \mathbb{A}_1,  \mathbb{A}_1,\mathbb{A}_2)},
\end{xy}$}
&  
\scalebox{0.7}{$\begin{xy}
(0, 0)*+{1^-}="E61", (10, 0)*+{2^-}="E62", (20, 0)*+{3^-}="E63", (30, 0)*+{5^-}="E64", (40, 0)*+{6^+}="E65", (20, 10)*+{4^-}="E66", 
{"E64" \ar@{<-} "E65"},
(20, -8)*+{(\mathbb{A}_1,\mathbb{A}_1,\mathbb{A}_1, \mathbb{A}_1, \mathbb{A}_2)},
\end{xy}$}\\ \hline
& $4$ & $10$ & $4$ & $2$  \\\hline \hline 
\end{tabular}
}
\end{center}
\caption{A half of single quivers of the double quiver of $\mathbb{E}_6$.}
\label{Fig.E6}
\end{figure}

\begin{example} \label{example-E6}
\begin{enumerate}[\rm (1)]
\item Let $Q$ be a quiver whose underlying graph $\Delta$ is one of Dynkin graphs of type $\mathbb{A}$, $\mathbb{D}$ and $\mathbb{E}$.
Let $A$ be the trivial extension of the path algebra $\mathbf{k}Q$ of $Q$ by a minimal co-generator.
It is easy to see that $A$ is a symmetric RCZ algebra if $Q$ is bipartite.
In this case, the Gabriel quiver of $A$ is precisely the double quiver $Q_{\Delta}$ of $\Delta$, in other words, the graph of $A$ is $\Delta$.
On the other hand, $Q^{\rm op}$ also determines the symmetric RCZ algebra, which is naturally isomorphic to $A$. 
\item Let $\Delta=\mathbb{E}_{6}$ and let $A$ be the symmetric RCZ algebra obtained as in (1).
In Figure \ref{Fig.E6}, we describe single quivers of $Q:=Q_{\mathbb{E}_6}$ associated to maps $\epsilon$ with $\epsilon(6)=+1$.
Here, the notation $i^{\sigma}$ denotes the vertex $i$ with $\epsilon(i)=\sigma \in \{\pm1\}$.
Using the Corollary \ref{number by graph}, we find that there are $1700$ isomorphism classes of basic two-term tilting complexes over $A$ as in the list of Theorem \ref{theorem2}.
\end{enumerate}
\end{example}

\section{Proof of Main Theorem} \label{main theorem}

In this section, we prove Theorems \ref{theorem1} and \ref{theorem2}.
Throughout this section, $G$ is a finite connected graph with no loops.

\subsection{Proof of Theorem \ref{theorem1}}

By Corollary \ref{number by graph}(1), the proof is completed with the following proposition.

\begin{proposition}\label{tothm1}
Let $G$ be a connected finite graph with no loops. Then the graph $G_{\epsilon}$ is a disjoint union of Dynkin graphs of type $\mathbb{A}$, $\mathbb{D}$ and $\mathbb{E}$ for every map $\epsilon \colon G_{0} \to\{\pm 1\}$ if and only if $G$ is one of the list in Theorem \ref{theorem1}. 
\end{proposition}

In the following, we give a proof of Proposition \ref{tothm1} by removing extended Dynkin graphs from the collection $G_{\epsilon}$ of subgraphs of $G$. 
We start with removing extended Dynkin graphs of type $\widetilde{\mathbb{A}}$.
A graph is called an \emph{$n$-cycle} if it is a cycle with exactly $n$ vertices.
In particular, it is called an \emph{odd-cycle} if $n$ is odd, and an \emph{even-cycle} if $n$ even.

\begin{lemma}\label{remove-ext-A}
The following statements are equivalent:
\begin{enumerate}[\rm (1)]
\item There exists a map $\epsilon \colon G_{0} \to\{\pm 1\}$ such that $G_{\epsilon}$ contains an extended Dynkin graph of type $\widetilde{\mathbb{A}}$ as a subgraph.
\item $G$ contains an even-cycle as a subgraph.
\end{enumerate}
\end{lemma}
\begin{proof}
(2)$\Rightarrow$(1): Let $G'$ be a subgraph of $G$ which is an even-cycle.
Since an even-cycle is a bipartite graph, there exists a map $\epsilon \colon G_{0} \to \{\pm 1\}$ such that the underlying graph of $G_{\epsilon}$ contains $G'$ as a subgraph.
Hence the assertion follows.

(1)$\Rightarrow$(2): Assume that for some map $\epsilon\colon G_{0}\to \{ \pm 1\}$, the graph $G_{\epsilon}$ contains an extended Dynkin graph $G'$ of type $\widetilde{\mathbb{A}}$. Since $G_{\epsilon}$ is bipartite, so is $G'$.
Hence $G'$ is an even-cycle and a subgraph of $G$.
This finishes the proof. 
\end{proof}

By Lemma \ref{remove-ext-A}, we may assume that $G$ contains no even-cycle as a subgraph.
In particular, $G$ has no multiple edges.
We give a connection between our graphs $G_{\epsilon}$ and subtrees of $G$.
Recall that a \emph{subtree} of $G$ is a connected subgraph of $G$ without cycles.

\begin{proposition}\label{subtree-bipartite}
Assume that $G$ contains no even-cycle as a subgraph. 
Let $G'$ be a connected graph.
Then the following statements are equivalent.
\begin{enumerate}[\rm (1)]
\item There exists a map $\epsilon \colon G_{0} \to\{ \pm 1\}$ such that $G_{\epsilon}$ contains $G'$ as a subgraph.
\item $G'$ is a subtree of $G$.
\end{enumerate}

In particular, there exists a naturally two-to-one correspondence between the set of connected graphs of the form $G_{\epsilon}$ and the set of subtrees of $G$.
\end{proposition}
\begin{proof}
(2)$\Rightarrow$(1) is clear.
We show (1)$\Rightarrow$(2).
Since $G$ has no even-cycle as a subgraph, then $G_{\epsilon}$ is tree by Lemma \ref{remove-ext-A}.
Since $G'$ is a subgraph of $G$, any subgraph of $G'$ is a subtree of $G$. 
\end{proof}

For a tree, we have the following result.
\begin{corollary}\label{Dynkin-case}
Assume $G$ is a tree.
Then the graph $G_{\epsilon}$ is a disjoint union of Dynkin graphs of type $\mathbb{A}$, $\mathbb{D}$ and $\mathbb{E}$ for each map $\epsilon \colon G_{0} \to \{\pm 1\}$ if and only if $G$ is a Dynkin graph of type $\mathbb{A}$, $\mathbb{D}$ and $\mathbb{E}$.
\end{corollary}
\begin{proof}
It is well known that $G$ is a Dynkin graph if and only if all subtrees of $G$ are Dynkin graphs.
The assertion follows from Proposition \ref{subtree-bipartite}. 
\end{proof}

We remove extended Dynkin graphs of type $\widetilde{\mathbb{D}}$.
Assume that $G$ contains at least two odd-cycles.
Then there exists a subtree $G'$ of $G$ such that $G'$ is an extended Dynkin graph of type $\widetilde{\mathbb{D}}$.
Moreover, by Proposition \ref{subtree-bipartite}, there exists a map $\epsilon \colon G_{0} \to \{\pm 1\}$ such that $G_\epsilon$ contains an extended Dynkin graph of type $\widetilde{\mathbb{D}}$ as a subgraph. 
Hence we may assume that $G$ contains at most one odd-cycle.
By Corollary \ref{Dynkin-case}, it is enough to consider the case where $G$ contains exactly one odd-cycle.
Namely, $G$ consists of an odd-cycle such that each vertex $v$ in the odd-cycle is attached to a tree $T_{v}$.
\begin{align}
\xymatrix@C=4mm@R=3mm{
\bullet\ar@{-}[r]&\bullet\ar@{-}[r]&v_{1}\ar@{-}[dr]\ar@{-}[dl]\ar@{-}[r]&\bullet&\bullet\ar@{-}[d]&\\
&v_{2}\ar@{-}[rr]&&v_{3}\ar@{-}[r]&\bullet\ar@{-}[r]&\bullet\\
}\notag
\end{align}

\begin{lemma}\label{remove-ext-D}
Fix an integer $k\ge 1$ and $n:=2k+1$.
Assume that $G$ consists of an $n$-cycle such that each vertex $v$ in the $n$-cycle is attached to a tree $T_{v}$.
Then the following statements are equivalent:
\begin{enumerate}[\rm (1)]
\item There exists a map $\epsilon \colon G_{0} \to\{ \pm 1\}$ such that $G_{\epsilon}$ contains an extended Dynkin graph of type $\widetilde{\mathbb{D}}$ as a subgraph.
\item $G$ contains an extended Dynkin graph of type $\widetilde{\mathbb{D}}$ as a subgraph. 
\item $G$ satisfies one of the following conditions.
\begin{enumerate}[\rm (a)]
\item There is a vertex $v$ in the $n$-cycle such that the degree is at least four.
\item There is a vertex $v$ in the $n$-cycle such that the degree is exactly three and $T_{v}$ is not a Dynkin graph of type $\mathbb{A}$. 
\item $k\ge 2$ and there are at least two vertices in the $n$-cycle such that the degrees are at least three.
\end{enumerate}
\end{enumerate}
\end{lemma}
\begin{proof}
(1)$\Leftrightarrow$(2) follows from Proposition \ref{subtree-bipartite}.
Moreover, we can easily check (2)$\Leftrightarrow$(3) because $\widetilde{\mathbb{D}}_{4}$ has exactly one vertex whose degree is exactly four and $\widetilde{\mathbb{D}}_{l}$ ($l\ge 5$) has exactly two vertices whose degree are exactly three.
\end{proof}

Fix an integer $k \ge 1$ and $n:=2k+1$.
By Lemma \ref{remove-ext-D}, we may assume that $G$ is one of the following graphs:

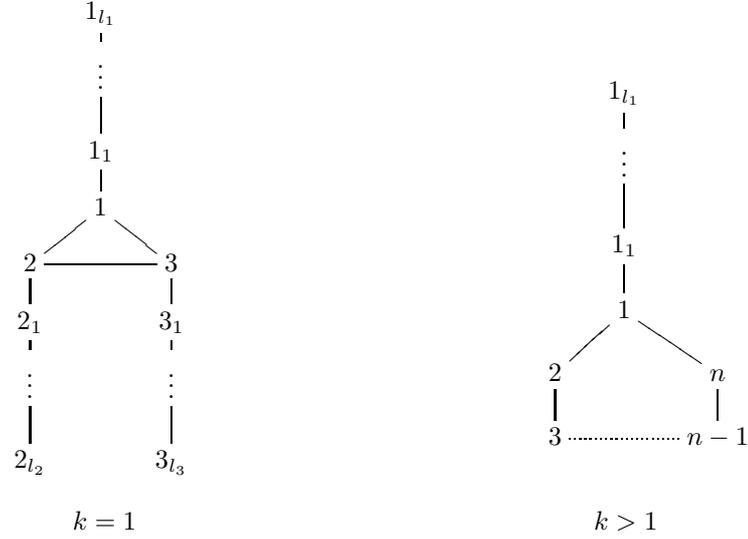
\begin{align}
\begin{picture}(400,200)(0,0)
\put(95,5){$k=1$}
\put(290,5){$k\ge 2$}
\put(70,200){\xymatrix@C=3mm@R=3mm{
&1_{l_{1}}\ar@{-}[d]&&\\
&\vdots\ar@{-}[d]&&\\
&1_{1}\ar@{-}[d]&&\\
&1\ar@{-}[rd]\ar@{-}[ld]&&\\
2\ar@{-}[rr]\ar@{-}[d]&&3\ar@{-}[d]\\
2_{1}\ar@{-}[d]&&3_{1}\ar@{-}[d]\\
\vdots\ar@{-}[d]&&\vdots\ar@{-}[d]\\
2_{l_{2}}&&3_{l_{3}}
}}
\put(270,170){\xymatrix@C=4mm@R=4mm{
&1_{l_{1}}\ar@{-}[d]&\\
&\vdots\ar@{-}[d]&\\
&1_{1}\ar@{-}[d]&\\
&1\ar@{-}[rd]\ar@{-}[ld]&\\
2\ar@{-}[d]&&n\ar@{-}[d]\\
3\ar@{.}[rr]&&{n-1}
}}
\end{picture}\notag
\end{align}

Finally, we remove extended Dynkin graphs of type $\widetilde{\mathbb{E}}$. 

\begin{lemma}\label{remove-ext-E}
Fix an integer $k \ge 1$ and $n:=2k+1$. 
\begin{enumerate}[\rm (1)]
\item Assume that $k=1$. 
The following graphs $(\mathrm{i})$, $(\mathrm{ii})$ and $(\mathrm{iii})$ are the minimal graphs containing an extended Dynkin graph of type $\widetilde{\mathbb{E}}$.
\begin{align}
\begin{picture}(400,160)(0,0)
\put(25,150){\textnormal{(i)}}
\put(30,150){\xymatrix@C=3mm@R=3mm{
&1_{1}\ar@{-}[d]&\\
&1\ar@{-}[rd]\ar@{-}[ld]&\\
2\ar@{-}[d]\ar@{-}[rr]&&3\ar@{-}[d]\\
2_{1}&&3_{1}\ar@{-}[d]\\
&&3_{2}
}}
\put(145,150){\textnormal{(ii)}}
\put(265,150){\textnormal{(iii)}}
\put(150,150){\xymatrix@C=3mm@R=3mm{
&1\ar@{-}[rd]\ar@{-}[ld]&\\
2\ar@{-}[rr]\ar@{-}[d]&&3\ar@{-}[d]\\
2_{1}\ar@{-}[d]&&3_{1}\ar@{-}[d]\\
2_{2}&&3_{2}\ar@{-}[d]\\
&&3_{3}
}}
\put(270,150){\xymatrix@C=3mm@R=3mm{
&1\ar@{-}[rd]\ar@{-}[ld]&\\
2\ar@{-}[rr]\ar@{-}[d]&&3\ar@{-}[d]\\
2_{1}&&3_{1}\ar@{-}[d]\\
&&3_{2}\ar@{-}[d]\\
&&3_{3}\ar@{-}[d]\\
&&3_{4}\ar@{-}[d]\\
&&3_{5}
}}
\end{picture}\notag
\end{align}
\item Assume that $k\ge 2$.
The following graphs $(\mathrm{iv})$ and $(\mathrm{v})$ are 
the minimal graphs containing an extended Dynkin graph of type $\widetilde{\mathbb{E}}$.
\begin{align}
\begin{picture}(300,145)(0,0)
\put(60,5){\textnormal{(iv)} $k= 2$}
\put(210,5){\textnormal{(v)} $k\ge 3$}
\put(45,130){\xymatrix@C=4mm@R=4mm{
&1_{2}\ar@{-}[d]&\\
&1_{1}\ar@{-}[d]&\\
&1\ar@{-}[rd]\ar@{-}[ld]&\\
2\ar@{-}[d]&&n\ar@{-}[d]\\
3\ar@{.}[rr]&&{n-1}
}}
\put(195,130){\xymatrix@C=4mm@R=4mm{
&1_{1}\ar@{-}[d]&\\
&1\ar@{-}[rd]\ar@{-}[ld]&\\
2\ar@{-}[d]&&n\ar@{-}[d]\\
3\ar@{-}[d]&&n-1\ar@{-}[d]\\
4\ar@{.}[rr]&&n-2
}}
\end{picture}\notag
\end{align}
\end{enumerate}
\end{lemma}
\begin{proof}
We can easily find extended Dynkin graphs $\widetilde{\mathbb{E}}_{6}$, $\widetilde{\mathbb{E}}_{7}$ and $\widetilde{\mathbb{E}}_{8}$ in the graphs above. 
\end{proof}

Now we are ready to prove Proposition \ref{tothm1}.

\begin{proof}[Proof of Proposition \ref{tothm1}]
If $G$ is a tree, then the assertion follows from Corollary \ref{Dynkin-case}.
We assume that $G$ is not a tree. 
By Lemma \ref{remove-ext-A}, we may assume that $G$ does not contain even-cycles as subgraphs.
Then $G$ does not contain extended Dynkin graphs as subgraphs if and only if $G$ is one of the following classes:
\begin{itemize}
\item $(\mathrm{I}_{n})_{n\ge 4}$ in Theorem \ref{theorem1}(2), 
\item proper connected non-tree subgraphs appearing in Lemma \ref{remove-ext-E}(i)--(v).
\end{itemize}
The second class coincides with the graphs $(\widetilde{\mathbb{A}}_{n-1})_{n:{\rm odd}}$, $(\mathrm{I}_{n})_{4\le n \le 8}$, $(\mathrm{II}_{n})_{5\le n\le 8}$, (III), (IV) and (V) in Theorem \ref{theorem1}(2).
Hence the assertion follows from Proposition \ref{subtree-bipartite}.
\end{proof}

We finish this subsection with proof of Theorem \ref{theorem1}.

\begin{proof}[Proof of Theorem \ref{theorem1}] 
The result follows from Corollary \ref{number by graph}(1) and Proposition \ref{tothm1}.
\end{proof}

\subsection{Proof of Theorem 1.2}
We just compute the number of two-term tilting complexes for each graph in the list of Theorem \ref{theorem1}. Our calculation is based on Theorem \ref{reduced ver} and Corollary \ref{number by graph}. For our purpose, we assume that $G$ is a graph appearing in the list of Theorem \ref{theorem1} and let $A$ be a basic connected finite dimensional symmetric RCZ algebra whose graph is $G$. 

Keeping above notations, we determine the number $\#\twotilt A$, or equivalently, $||G||$ in Definition \ref{number GG}.  
First, for types $\mathbb{A}$ and $\widetilde{\mathbb{A}}$, the number is already computed by \cite{Aoki18}: 
 
\begin{proposition} \cite[Theorem 1.2]{Aoki18} \label{type A}
The following equality holds.
\begin{align}
\# \twotilt A = 
\begin{cases}
\binom{2n}{n} &\text{if $G=\mathbb{A}_{n}$,}\\
2^{2n-1}  &\text{if $G=\widetilde{\mathbb{A}}_{n-1}$ for odd $n$.}
\end{cases}\notag
\end{align}
\end{proposition}

Secondly, we consider the case where $G$ is a Dynkin graph of type $\mathbb{D}$.
For simplicity, let $c_{0}=1$, $c_{l}:=\binom{2l}{l}$ for each $l\ge 1$. 
Then we have $||\mathbb{A}_l|| = c_l$ for all $l\geq 1$ by Proposition \ref{type A}. In addition, let $||\mathbb{A}_0||:=2$.   

\begin{proposition}\label{type D}
Let $n\ge 4$ and $G=\mathbb{D}_n$. Then we have
\begin{align}
\# \twotilt A = 6\cdot 4^{n-2} - 2 c_{n-2}.\notag
\end{align}
\end{proposition}
\begin{proof}
Let $G$ be a graph as follows.
\begin{align}
\xymatrix@R=1mm{
1&&&&&&&\\
&3\ar@{-}[lu]\ar@{-}[ld]\ar@{-}[r]&4\ar@{-}[r]&\cdots\ar@{-}[r]&n.\\
2&&&&&&&
}\notag
\end{align}
By Corollary \ref{number by graph}, we have
\begin{equation} \label{for D}
\# \twotilt A =2 \cdot \sum_{\substack{\epsilon\colon Q_{0} \rightarrow \{ \pm 1\} \\ \epsilon(3)=+1}} |{G}_{\epsilon}|.
\end{equation}
We study the right hand side of (\ref{for D}). Let $M$ be the set of maps $\epsilon \colon G_{0} \rightarrow \{\pm1\}$ such that $\epsilon(3)=+1$.  
Clearly, $M$ is a disjoint union of the following subsets: 
\begin{itemize}
\item $M_{1}:=\{ \epsilon\in M \mid \epsilon(1)=\epsilon(2)=\epsilon(3) \}$.
\item $M_{2}:=\{ \epsilon\in M \mid \epsilon(1)=-\epsilon(2)=\epsilon(3) \}$.
\item $M_{3}:=\{ \epsilon\in M \mid -\epsilon(1)=\epsilon(2)=\epsilon(3) \}$.
\item $M_{4}:=\{ \epsilon\in M \mid -\epsilon(1)=-\epsilon(2)=\epsilon(3)=\epsilon(4) \}$.
\item $M_{5}:=\{ \epsilon\in M \mid -\epsilon(1)=-\epsilon(2)=\epsilon(3)=-\epsilon(4) \}=\bigsqcup_{t=4}^{n}M_{5}(t)$, where
\begin{align}
M_{5}(t):=\Bigl\{ \epsilon\in M_{5}\ \Bigl.\Bigr|\ t=\min\{ 4 \le j \le n \mid \epsilon(j)=\epsilon(j+1)\} \Bigr\}. \notag
\end{align}
\end{itemize}
From now, we compute $\mathsf{n}(i):=\sum_{\epsilon \in M_i} |G_{\epsilon}|$ for each $i\in\{ 1,\ldots, 5\}$.
In the following, the notation $\xymatrix{i\ar@{~}[r]&j}$ is replaced by an edge connecting $i$ and $j$ if $\epsilon(i)\neq\epsilon(j)$, otherwise nothing between them.

(i) Let $\epsilon \in M_{1}$. Then $G_{\epsilon}$ is given by 
\begin{align}
\xymatrix@R=1mm{
1&&&&&\\
&3\ar@{~}[r]&4\ar@{~}[r]&\cdots\ar@{~}[r]&n-1\ar@{~}[r]&n.\\
2&&&&&
}\notag
\end{align}
Let $G'$ be the subgraph of $G$ obtained by removing the vertices $\{1,2\}$. Then we have $|G_{\epsilon}| = |G'_{\epsilon|_{\{3,\ldots,n\}}}|$. 
Since $G'$ is a Dynkin graph of type $\mathbb{A}_{n-2}$, we obtain 
\[
  2\mathsf{n}(1)= 2 \cdot \sum_{\substack{\epsilon\colon G_0' \to \{\pm1\} \\ \epsilon(3)=+1}} |G'_{\epsilon}| =||\mathbb{A}_{n-2}|| = c_{n-2}
\] 
where the last equality follows from Proposition \ref{type A}.

By an argument similar to (1), we can calculate other cases.

(ii) For each $\epsilon\in M_{2}$, the graph $G_{\epsilon}$ is given by
\begin{align}
\xymatrix@R=1mm{
1&&&&&\\
&3\ar@{-}[ld]\ar@{~}[r]&4\ar@{~}[r]&\cdots\ar@{~}[r]&n-1\ar@{~}[r]&n.\\
2&&&&&
}\notag
\end{align}
Then we can check $2\mathsf{n}(2)=||\mathbb{A}_{n-1}||-||\mathbb{A}_{n-2}||=c_{n-1}-c_{n-2}$.

(iii) By the symmetry of $G$, we have $\mathsf{n}(3)=\mathsf{n}(2)$.

(iv) Let $\epsilon\in M_{4}$. Then $G_{\epsilon}$ is described as 
\begin{align}
\xymatrix@R=1mm{
1&&&&&\\
&3\ar@{-}[lu]\ar@{-}[ld]&4\ar@{~}[r]&\cdots\ar@{~}[r]&n-1\ar@{~}[r]&n.\\
2&&&&&
}\notag
\end{align}
Thus we find that $2\mathsf{n}(4)= |\mathbb{A}_3| \cdot ||\mathbb{A}_{n-3}|| = 5c_{n-3}$.

(v) For $\epsilon\in M_{5}(t)$, the graph $G_{\epsilon}$ is given by
\begin{align}
\xymatrix@R=1mm{
1&&&&&&&\\
&3\ar@{-}[lu]\ar@{-}[ld]\ar@{-}[r]&4& \ar@{-}[l]\cdots\ar@{-}[r]&t&t+1\ar@{~}[r]&\cdots\ar@{~}[r]&n.\\
2&&&&&&&
}\notag
\end{align}
Then we obtain 
\begin{align}
2\mathsf{n}(5)
&=\sum_{t=4}^{n}|\mathbb{D}_t| \cdot ||\mathbb{A}_{n-t}||
=\frac{3n-4}{2n}c_{n-1}\cdot 2c_{0}+\sum_{t=4}^{n-1}\frac{3t-4}{2t}c_{t-1}c_{n-t}\notag\\
&=\frac{3n-4}{2n}c_{n-1}+\sum_{t=4}^{n}\frac{3t-4}{2t}c_{t-1}c_{n-t}.\notag
\end{align}

To finish the proof, we need the following lemma.
\begin{lemma}  \label{binom numbers}
For any positive integer $n$, the following equalities hold:
\begin{enumerate}[(1)]
\item $\displaystyle{\sum_{t=1}^{n}c_{t-1}c_{n-t}=4^{n-1}}$.
\item $\displaystyle{\sum_{t=1}^{n}\frac{1}{t}c_{t-1}c_{n-t}=\frac{1}{2}c_{n}}$.
\end{enumerate}
\end{lemma}
\begin{proof}
The equality (1) is well-known. 
The equality (2) is obtained by
\begin{align}
\sum_{t=1}^{n}\frac{1}{t}c_{t-1}c_{n-t}
=\frac{n+1}{2}\sum_{t=1}^{n}C_{t-1}C_{n-t}
=\frac{n+1}{2}C_{n}=\frac{1}{2}c_{n},\notag
\end{align}
where $C_{n} := \frac{1}{n+1}c_{n}$ is the $n$-th Catalan number.
\end{proof}

By Lemma \ref{binom numbers}, we obtain the equality
\begin{align}
\sum_{t=4}^{n}\frac{3t-4}{2t}c_{t-1}c_{n-t}
&=\frac{3}{2}\sum_{t=4}^{n} c_{t-1}c_{n-t}-2\sum_{t=4}^{n}\frac{1}{t}c_{t-1}c_{n-t}\notag\\
&=\frac{3}{2}(4^{n-1} -c_{n-1} -2 c_{n-2} -6 c_{n-3}) -2(\frac{1}{2}c_{n}-c_{n-1}-c_{n-2}-2c_{n-3})\notag\\
&= 6\cdot 4^{n-2}-c_{n}+\frac{1}{2}c_{n-1}-c_{n-2}-5c_{n-3}.\notag
\end{align}
By (i)--(v), we have 
\begin{align}
\# \twotilt A 
&=c_{n-2}+2(c_{n-1}-c_{n-2})+5c_{n-3}+6\cdot 4^{n-2}-c_{n}+\frac{2n-2}{n}c_{n-1}-c_{n-2}-5c_{n-3}\notag\\
&=6\cdot 4^{n-2}-c_{n}+\frac{4n-2}{n}c_{n-1}-2c_{n-2}\notag\\
&=6\cdot 4^{n-2}-2c_{n-2},\notag
\end{align}
where the last equality follows from $c_{n}=\frac{2(2n-1)}{n}c_{n-1}$.
\end{proof}

Thirdly, we give an enumeration for type ($\mathrm{I}$). The number is obtained by using the result on type $\mathbb{D}$.

\begin{proposition} \label{type I}
If $G=\mathrm{I}_{n}$, then we have 
\begin{align}
\#\twotilt A = 6\cdot 4^{n-2} + 2c_{n} -4c_{n-1}-4c_{n-2}.\notag
\end{align}
\end{proposition}
\begin{proof}
Let $G$ be a graph as follows.
\begin{align}
\xymatrix@R=1mm{
1\ar@{-}[dd]&&&&&&&\\
&3\ar@{-}[lu]\ar@{-}[ld]\ar@{-}[r]&4\ar@{-}[r]&\cdots\ar@{-}[r]&n.\\
2&&&&&&&
}\notag
\end{align}
By using a similar method of the proof of proposition \ref{type D}, we calculate the right-hand side of
\begin{align}
\# \twotilt A = 2 \sum_{\substack{\epsilon\colon G_{0} \rightarrow \{ \pm 1\} \\ \epsilon(3)=+1}} |G_{\epsilon}|.\notag
\end{align}
Let $M$ and $M_{i}$ $(1\leq i \leq 5)$ be sets of maps given in the proof of Proposition \ref{type D} and $\mathsf{m}(i):=\sum_{\epsilon\in M_{i}} |G_{\epsilon}|$.
For each map $\epsilon\in M_{1}\sqcup M_{4}\sqcup M_{5}$, we have $G_{\epsilon}=(\mathbb{D}_n)_{\epsilon}$.
Hence for each $i\in \{1,4,5\}$, we have
\begin{align}
\mathsf{m}(i)=\sum_{\epsilon\in M_{i}} |G_{\epsilon}|=\sum_{\epsilon\in M_{i}} |(\mathbb{D}_n)_{\epsilon}| = \mathsf{n}(i).\notag
\end{align}
Since $\mathsf{m}(2)=\mathsf{m}(3)$ holds by the symmetry of $G$, we have only to calculate $\mathsf{m}(2)$.
For each map $\epsilon\in M_{2}$, the graph $G_{\epsilon}$ is given by 
\begin{align}
\xymatrix@R=1mm{
1\ar@{-}[dd]&&&&&\\
&3\ar@{-}[ld]\ar@{~}[r]&4\ar@{~}[r]&\cdots\ar@{~}[r]&n-1\ar@{~}[r]&n.\\
2&&&&&
}\notag
\end{align}
Then the calculation of $\mathsf{m}(2)$ is reduced to that of Dynkin graphs of type $\mathbb{A}$.
In fact, let $G'$ be the Dynkin graph $\mathbb{A}_{n}$.
Then we have 
\begin{align}
\mathsf{m}(2)
&= \sum_{\substack{\epsilon\colon G'_{0}\to \{ \pm 1\}\\ \epsilon(3)=+1}} |G'_{\epsilon}|- \sum_{\substack{\epsilon\colon G'_{0}\to \{ \pm 1\}\\ \epsilon(2)=\epsilon(3)=+1}} |G'_{\epsilon}| - 
\sum_{\substack{\epsilon\colon G'_{0}\to \{ \pm 1\} \\ -\epsilon(1)=-\epsilon(2)=\epsilon(3)=+1}} |G'_{\epsilon}|\notag \\
&=\frac{1}{2}||\mathbb{A}_{n}|| - \frac{1}{4} ||\mathbb{A}_{2}|| \cdot ||\mathbb{A}_{n-2}|| - \mathsf{n}(2). \notag\\
&=\frac{1}{2}c_{n}-\frac{3}{2}c_{n-2}-\mathsf{n}(2).\notag
\end{align}
Therefore we obtain
\begin{align}
\#\twotilt A
&= 2(\mathsf{m}(1)+\mathsf{m}(2)+\mathsf{m}(3)+\mathsf{m}(4)+\mathsf{m}(5))\notag\\
&= 2(\mathsf{n}(1)+2\mathsf{n}(2)+\mathsf{n}(4)+\mathsf{n}(5))-4\mathsf{n}(2)+4\mathsf{m}(2)\notag\\
&= || \mathbb{D}_{n}|| - 4\mathsf{n}(2) + 4\mathsf{m}(2) \notag\\
&= 6\cdot 4^{n-2} - 2c_{n-2} -4\mathsf{n}(2) + 2c_{n}-6c_{n-2}-4\mathsf{n}(2) \notag\\
&= 6\cdot 4^{n-2} + 2c_{n} - 4c_{n-1} - 4c_{n-2}. \notag
\end{align}
This finishes the proof.
\end{proof}

For the remained finite series $\mathbb{E}$, (II), (III), (IV) and (V), 
we just compute the number by using the formula (\ref{number G}) in Corollary \ref{number by graph}(2).

\begin{proposition} \label{type sporadic}
For each case \textnormal{$\mathbb{E}$, (II), (III), (IV)} and \textnormal{(V)}, 
the number $\twotilt A_G$ is given by the table of Theorem \ref{theorem2}.
\end{proposition}
\begin{proof}
The number for $\mathbb{E}_6$ is shown in Example \ref{example-E6}(2) and the others are similar. The detail is left to the reader. 
\end{proof}

\subsection*{Acknowledgements}
T.~Adachi is supported by JSPS KAKENHI Grant Number JP17J05537.
T.~Aoki is supported by JSPS KAKENHI Grant Number JP19J11408.
The authors thank to M.~Konishi for helpful discussions about the proof of Proposition \ref{tothm1}.


\begin{thebibliography}{99}

\bibitem{Adachi16a} 
T.~Adachi, \emph{The classification of $\tau$-tilting modules over Nakayama algebras}, J. Algebra {\bf 452} (2016), 227--262.

\bibitem{Adachi16b}
T.~Adachi, \emph{Characterizing $\tau$-tilting finite algebras with radical square zero}, Proc. Amer. Math. Soc. {\bf 144} (2016), no.~11, 4673--4685.

\bibitem{AIR}
T.~Adachi, O.~Iyama, I.~Reiten, \emph{$\tau$-tilting theory}, Compos. Math. {\bf 150} (2014), no.~3, 415--452.

\bibitem{AI}
T.~Aihara, O.~Iyama, \emph{Silting mutation in triangulated categories},  J. Lond. Math. Soc. (2) {\bf 85} (2012), no.~3, 633--668.

\bibitem{Aoki18}
T.~Aoki, \emph{Classifying torsion classes for algebras with radical square zero via sign decomposition}, arXiv:1803.03795v2.

\bibitem{BR}
E.~Barnard, N.~Reading, \emph{Coxeter-biCatalan combinatorics},  J. Algebraic Combin. {\bf 47} (2018), no.~2, 241--300.

\bibitem{Benson08}
D.J.~Benson, \emph{Resolutions over symmetric algebras with radical cube zero}, J. Algebra {\bf 320} (2008), no.~1, 48--56.

\bibitem{CL}
S.~Cautis, A.~Licata, \emph{Heisenberg categorification and Hilbert schemes}, Duke Math. J. {\bf 161} (2012), no.~13, 2469--2547.

\bibitem{DIRRT}
L.~Demonet, O.~Iyama, N.~Reading, I.~Reiten, H.~Thomas, \emph{Lattice theory of torsion classes}, arXiv:1711.01785.

\bibitem{EJR}
F.~Eisele, G.~Janssens, T.~Raedschelders, \emph{A reduction theorem for $\tau$-rigid modules}, Math. Z. {\bf 290} (2018), no.~3-4, 1377--1413.

\bibitem{ES}
K.~Erdmann, \O.~Solberg, \emph{Radical cube zero weakly symmetric algebras and support varieties}, J. Pure Appl. Algebra {\bf 215} (2011), no.~2, 185--200.

\bibitem{GSa}
E.L.~Green, S.~Schroll, \emph{Multiserial and special multiserial algebras and their representations}, Adv. Math. {\bf 302} (2016), 1111--1136.

\bibitem{HK}
R.S.~Huerfano, M.~Khovanov, \emph{A category for the adjoint representation}, J. Algebra {\bf 246} (2001), no.~2, 514--542.

\bibitem{ONFR}
M.A.A.~Obaid, S.K.~Nauman, W.M.~Fakieh, C.M.~Ringel, \emph{The number of support-tilting modules for a Dynkin algebra}, J. Integer Seq. {\bf 18} (2015), no.~10, Article 15.10.6.

\bibitem{Okuyama86}
T.~Okuyama, \emph{On blocks of finite groups with radical cube zero}, Osaka J. Math. {\bf 23} (1986), no.~2, 461--465.

\bibitem{Rickard89der}
J.~Rickard, \emph{Morita theory for derived categories}, J. London Math. Soc. (2) {\bf 39} (1989), no.~3, 436--456.

\bibitem{Seidel08}
P.~Seidel, \emph{Fukaya categories and Picard-Lefschetz theory}, Zurich Lectures in Advanced Mathematics. European Mathematical Society (EMS), Z\"{u}rich, 2008.

\bibitem{Zhang13}
X.~Zhang, \emph{$\tau$-rigid modules for algebras with radical square zero}, arXiv:1211.5622v5.

\end{thebibliography}
\end{document}